\documentclass[reqno, 11pt, a4paper]{amsart}

\usepackage{latexsym,ifthen,xspace}
\usepackage{amsmath,amssymb,amsthm}
\usepackage{enumerate, calc}
\usepackage[colorlinks=true, pdfstartview=FitV, linkcolor=blue,
  citecolor=blue, urlcolor=blue, pagebackref=false]{hyperref}
\usepackage[text={33pc,605pt},centering]{geometry}    
\usepackage{graphicx}

\newtheorem{theorem}{Theorem}[section]
\newtheorem{lemma}[theorem]{Lemma}
\newtheorem{prop}[theorem]{Proposition}
\newtheorem{assumption}[theorem]{Assumption}
\newtheorem{corro}[theorem]{Corollary}

\theoremstyle{definition}
\newtheorem{definition}[theorem]{Definition}

\theoremstyle{remark}
\newtheorem{remark}[theorem]{Remark}

\numberwithin{equation}{section}

\DeclareMathAlphabet{\mathsl}{OT1}{cmss}{m}{sl}
\SetMathAlphabet{\mathsl}{bold}{OT1}{cmss}{bx}{sl}

\newcommand{\al}{\ensuremath{\alpha}}

\newcommand{\ga}{\ensuremath{\gamma}}

\newcommand{\ka}{\ensuremath{\kappa}}

\newcommand{\si}{\ensuremath{\sigma}}

\newcommand{\om}{\ensuremath{\omega}}

\newcommand{\ve}{\ensuremath{\varepsilon}}

\newcommand{\vp}{\ensuremath{\varphi}}

\newcommand{\Si}{\ensuremath{\Sigma}}

\newcommand{\Om}{\ensuremath{\Omega}}
\newcommand{\cD}{\ensuremath{\mathcal D}}
\newcommand{\cE}{\ensuremath{\mathcal E}}
\newcommand{\cF}{\ensuremath{\mathcal F}}

\newcommand{\cL}{\ensuremath{\mathcal L}}

\newcommand{\bbR}{\ensuremath{\mathbb R}}

\newcommand{\bbZ}{\ensuremath{\mathbb Z}} 
\newcommand{\md}{\ensuremath{\mathrm{d}}}
\newcommand{\mD}{\ensuremath{\mathrm{D}}}

\newcommand{\norm}[3]{%
   \ensuremath{%
     \mathchoice{\big\lVert #1 \big\rVert}
     {\lVert #1 \rVert}
     {\lVert #1 \rVert}
     {\lVert #1 \rVert}_{\raisebox{-.0ex}{$\scriptstyle \ell^{\raisebox{.2ex}{$\scriptscriptstyle #2$}} (#3)$}}
   }
}

\newcommand{\Norm}[2]{%
  \ensuremath{%
    \mathchoice{\big\lVert #1 \big\rVert}
     {\lVert #1 \rVert}
     {\lVert #1 \rVert}
     {\lVert #1 \rVert}_{\raisebox{-.0ex}{$\scriptstyle #2$}}
  }
}

\DeclareMathOperator{\mean}{\mathbb{E}}
\DeclareMathOperator{\Mean}{\mathrm{E}}
\DeclareMathOperator{\prob}{\mathbb{P}}
\DeclareMathOperator{\Prob}{\mathrm{P}}

\DeclareMathOperator{\supp}{\mathrm{supp}}
\DeclareMathOperator{\sign}{\mathrm{sign}}

\newcommand{\ldef}{\ensuremath{\mathrel{\mathop:}=}}

\begin{document}

\title[QIP for the one-dimensional dynamic RCM]{Invariance Principle for the one-dimensional dynamic Random Conductance Model under Moment Conditions}


\author{Jean-Dominique Deuschel}
\address{Technische Universit\"at Berlin}
\curraddr{Strasse des 17. Juni 136, 10623 Berlin}
\email{deuschel@math.tu-berlin.de}
\thanks{}

\author{Martin Slowik}
\address{Technische Universit\"at Berlin}
\curraddr{Strasse des 17. Juni 136, 10623 Berlin}
\email{slowik@math.tu-berlin.de}
\thanks{}



\date{\today}

\dedicatory{}

\begin{abstract}
  Recent progress in the understanding of quenched invariance principles (QIP) for a continuous-time random walk on $\bbZ^d$ in an environment of dynamical random conductances is reviewed and extended to the $1$-dimensional case.  The law of the conductances is assumed to be ergodic with respect to time-space shifts and satisfies certain integrability conditions.
\end{abstract}

\maketitle

\tableofcontents

\section{Introduction}

\subsection{The model}
Consider the $d$-dimensional Euclidean lattice, $(\bbZ_d, E_d)$, for $d \geq 1$.  The vertex set, $V_d$, of this graph equals $\bbZ^d$ and the edge set, $E_d$, is given by the set of all non-oriented nearest neighbour bonds, i.e. $E_d \ldef \{ \{x,y\}: x,y \in \bbZ^d, |x-y|=1\}$.  We also write $x \sim y$ if $\{x,y\} \in E_d$.  The graph $(\bbZ^d, E_d)$ is endowed with a family $\om = \{\om_t(e) : e \in E_d, \, t\in \bbR \} \in \Om \ldef (0, \infty)^{\bbR \times E_d}$ of time-dependent, positive weights.  To simplify notation,  for $x, y \in \bbZ^d$ and $t\in \bbR$ we set $\om_t(x,y) = \om_t(y,x) = \om_t(\{x,y\})$ if $\{x,y\} \in E_d$ and $\om_t(x,y) = 0$ otherwise.  We introduce a \emph{time-space} shift $\tau_{s,z}$ by $(s, z) \in \bbR \times \bbZ^d$ through
\begin{align*}
  (\tau_{s, z\,} \om)_t(x,y) \;\ldef\; \om_{t+s}(x+z,y+z),
  \qquad \forall\, t \in \bbR, \{x, y\} \in E_d.
\end{align*}
Further, consider a probability measure, $\prob$, on the measurable space $(\Om, \cF)$ where $\cF$ denotes the Borel-$\si$-algebra on $\Om$, and we write $\mean$ to denote the corresponding expectation with respect to $\prob$.

We impose the following conditions on the probability measure $\prob$.
\begin{assumption}\label{ass:P}
  Assume that $\prob$ satisfies the following conditions:
  \begin{enumerate}[(i)]
  \item $\prob$ is ergodic and stationary with respect to time-space shifts, that is $\prob \circ \tau_{t,x}^{-1} = \prob$ for all $t \in \bbR$, $x \in \bbZ^d$ and $\prob[A] \in \{0,1\}$ for any $A \in \cF$ such that $\tau_{t,x}(A) = A$ for all $t \in \bbR$, $x \in \bbZ^d$.
  \item $\mean\big[\om_t(e)\big]< \infty$ and $\mean\big[\om_t(e)^{-1}\big] < \infty$ for all $e \in E_d$ and $t \in \bbR$.
  \end{enumerate}
\end{assumption}
\begin{remark}
  Note that time-space ergodicity assumption is quite general.  In particular, it includes as a special case the static situation, that is the conductances $\om$ are independent of time and $\prob$ is ergodic with respect to space shifts.
\end{remark}
For any fixed $\om \in \Om$, we introduce the following (time-dependent) measures $\mu_t^{\om}$ and $\nu_t^{\om}$ on $\bbZ^d$ that are defined by
\begin{align}
  \mu_t^{\om}(x)
  \;\ldef\;
  \sum_{y \sim x} \om_t(x,y)
  \qquad \text{and} \qquad
  \nu_t^{\om}(x)
  \;\ldef\;
  \sum_{y \sim x} \frac{1}{\om_t(x,y)},
  \qquad \forall\, t \in \bbR.
\end{align}
In addition, for any compact interval $I \subset \bbR$ and any finite $B \subset \bbZ^d$ let us define a locally time-space averaged norm for functions $u\!:\bbR \times \bbZ^d \to \bbR$ by
\begin{align*}
  \Norm{u}{p,q,I \times B}
  \;\ldef\;
  \bigg(
    \frac{1}{|I|}\,
    \int_I
      \bigg(
        \frac{1}{|B|}\, \sum_{x \in B} |u(t,x)|^p
      \bigg)^{\!\!q/p}
    \md t
  \bigg)^{\!\!1/q},
  \qquad p, q \in (0, \infty),
\end{align*} 
where $|I|$ and $|B|$ denotes the Lebesgue measure of the interval and the cardinality of the set $B$, respectively.  Further, we write $B(x,r) \ldef \{y \in \bbZ^d \,:\, |y-x|_1 \leq \lfloor r \rfloor\}$ to denote the closed ball with respect to the $\ell^1$-norm with center $x \in \bbZ^d$ and radius $r \geq 0$.

For any fixed realization $\om \in \Om$, we consider the time-inhomogeneous Markov process, $X = \{X_t : t \geq 0\}$ on $\bbZ^d$ in the random environment $\om$ generated by
\begin{align}
  (\cL_t^{\om} f)(x)
  \;\ldef\;
  \sum_{y \sim x} \om_t(x,y)\, \big(f(y) \,-\, f(x) \big).
\end{align}
For any $s \in \bbR$ and $x \in \bbZ^d$, the measure $\Prob_{s, x}^{\om}$ on $\cD(\bbR, \bbZ^d)$, the space of $\bbZ^d$-valued c\`{a}dl\`{a}g functions on $\bbR$, denotes the law of the process $X$ starting at time $s$ in $x$.  In order to construct this Markov process under the law $\Prob_{s,x}^{\om}$, we specify in the sequel its jump times $s < J_1 < J_2 < \ldots$ inductively.  For this purpose, let $\{Z_k : k \geq 1\}$ be a sequence of independent $\mathop{\mathrm{Exp}}(1)$-distributed random variables, and set $J_0 = s$ and $X_s = x$.  Suppose that for any $k \geq 1$ the process $X$ has already been constructed on $[s, J_k]$.  Then, $J_{k+1}$ is defined by 
\begin{align*}
  J_{k+1}
  \;\ldef\;
  \inf
  \Big\{
    t \geq 0 \,:\, \int_{J_k}^{J_k + t}\! \mu_u^{\om}(X_{J_k})\, \md u \geq Z_{k+1}
  \Big\},
\end{align*}
and at time $t = J_k$ the random walk $X$ jumps from $z = X_{J_k}$ to any of its neighboring vertices $y$ with probability $\om_{t}(z, y) / \mu_t^{\om}(z)$.
\begin{lemma}
  For $\prob$-a.e.\ $\om$, $\Prob^{\om}_{0,0}$-a.s.\ the process  $\{X_t : t\geq 0\}$ does not explode, that is there are only finitely many jumps in finite time.
\end{lemma}
\begin{proof}
  See \cite[Lemma~4.1]{ACDS16}.
\end{proof}

Note that the counting measure on $\bbZ^d$, independent of $t$, is an invariant measure for $X$.

\subsection{Main result}
We are interested in the long time behaviour of the random walk among time-dependent random conductances for $\prob$-almost every realization $\om$.  In particular, our aim is to prove a quenched invariance principle for the process $X$ in the following sense.
\begin{definition} \label{def:QFCLT}
  Set $X_t^{(n)} \ldef \frac{1}{n} X_{n^2 t}$, $t \geq 0$. We say that the \emph{Quenched Functional CLT} (QFCLT) or \emph{quenched invariance principle} (QIP) holds for $X$, if for every $T > 0$ and every bounded continuous function $F$ on the Skorohod space $\cD([0,T], \bbR^d)$, it holds that $\Mean_{0,0}^{\om}[F(X^{(n)})] \to \Mean_{0,0}^{\mathrm{BM}}[F(\Si \cdot W)]$ as $n \to \infty$ for $\prob$-a.e.\ $\om$, where $(W, \Prob_{\!0,0}^{\mathrm{BM}})$ is a Brownian motion on $\bbR^d$ starting at time $0$ in $0$ with deterministic covariance matrix $\Si^2 = \Si \cdot \Si^T$.
\end{definition}
%
For $d \geq 2$ the following result has been obtained recently in \cite{ACDS16}.
\begin{theorem}\label{thm:main:d}
  Suppose that $d \geq 2$ and Assumptions \ref{ass:P} holds.  For $t \in \bbR$ and $e \in E_d$ assume that $\mean[\om_t(e)^p] < \infty$ and $\mean[\om_t(e)^{-q}] < \infty$ for any $p, q \in [d/2, \infty]$ such that
  \begin{align}\label{eq:cond:d}
    \frac{1}{p-1} \,+\, \frac{1}{(p-1)q} \,+\, \frac{1}{q}
    \;<\;
    \frac{2}{d}.
  \end{align}
  Then, the QIP holds for $X$ with a deterministic, time-independent, non-degenerate covariance matrix $\Si^2$.
\end{theorem}
\begin{remark}
  For static conductances a QIP holds if $\mean[\om(e)^p] < \infty$ and $\mean[\om(e)^{-q}]<\infty$ for any $p, q \geq 1$ such that $1/p + 1/q < 2/d$, see \cite{ADS15}.
\end{remark}
\begin{remark}
  The assertion of Theorem~\ref{thm:main:d} can be extended to the case that the law of the conductances satisfies different integrability conditions in time and space: For any $p, p', q, q' \in [1, \infty]$ satisfying
  \begin{align}\label{eq:cond:int:d}
    \frac{1}{p} \cdot \frac{p'}{p'-1} \cdot \frac{q'+1}{q'} \,+\, \frac{1}{q}
    \;<\;
    \frac{2}{d}
  \end{align}
  assume that 
  \begin{align}
    \lim_{n \to \infty} \Norm{\mu^{\om}}{p, p', Q(n)} \;<\; \infty
    \qquad \text{and} \qquad
    \lim_{n \to \infty} \Norm{\nu^{\om}}{q, q', Q(n)} \;<\; \infty.
  \end{align}
  Then, the QIP holds for $X$.  In particular, the integrability condition for the static RCM can be recovered from \eqref{eq:cond:int:d} by choosing $p'=q'=\infty$.
\end{remark}
In the present paper, we focus on the one-dimensional case:
\begin{theorem}\label{thm:main:1d}
  Suppose that $d = 1$ and Assumptions~\ref{ass:P} holds.  For $t \in \bbR$ and $e \in E_1$ assume that $\mean[\om_t(e)^p] < \infty$ and $\mean[\om_t(e)^{-q}] < \infty$ for any $p, q \in [1, \infty)$ such that
  \begin{align}\label{eq:cond:1d}
    \frac{1}{p-1} \,+\, \frac{1}{(p-1)q} \;<\; 1.
  \end{align}
  Then, the QIP holds for $X$ with a deterministic, time-independent variance $\si^2 > 0$.
\end{theorem}
\begin{figure}
  \begin{center}
    \includegraphics[width=215pt]{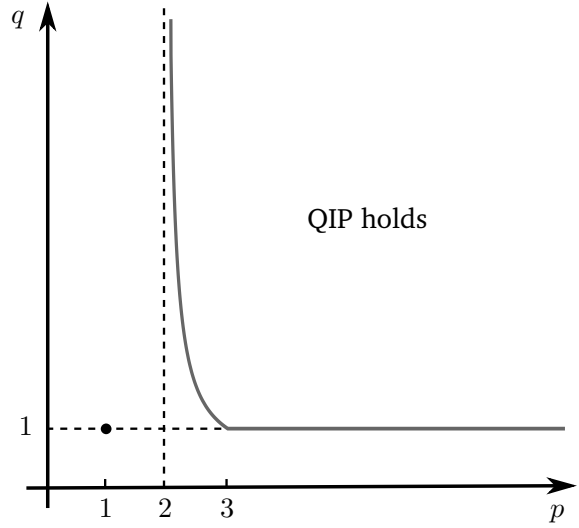}
  \end{center}
  \caption{Illustration of the area where the condition \eqref{eq:cond:1d} is satisfied}
   \label{fig:pq-area:1d}
\end{figure}
\begin{remark}
  Note that for all $p \in (2, \infty]$, \eqref{eq:cond:1d} is satisfied provided that $q > 1/(p-2)$.  Hence, in view Assumption~\ref{ass:P}, it suffices to choose $q=1$ for all $p \in (3, \infty]$, see also Fig.~\ref{fig:pq-area:1d}.  If one sets $d=1$ in \eqref{eq:cond:d}, we see that \eqref{eq:cond:1d} is equivalent to \eqref{eq:cond:d} for $q=1$ only.  This fact relies on the special shape of the Sobolev inequality in $d=1$ and will be explained below.
\end{remark}
\begin{remark}
  The assertion of Theorem~\ref{thm:main:1d} can also be extended to the case that the law of the conductances satisfies different integrability conditions in time and space: For any $p, p', q' \in [1, \infty]$ satisfying
  \begin{align}\label{eq:cond:int:1d}
    \frac{1}{p} \cdot \frac{p'}{p'-1} \cdot \frac{q'+1}{q'}
    \;<\;
    1
  \end{align}
  assume that 
  \begin{align}\label{eq:int:Exp:1d}
    \lim_{n \to \infty} \Norm{\mu^{\om}}{p, p', Q(n)} \;<\; \infty
    \qquad \text{and} \qquad
    \lim_{n \to \infty} \Norm{\nu^{\om}}{1, q', Q(n)} \;<\; \infty,
  \end{align}
  where $Q(n) \ldef [0, n^2] \times B(0, n)$.  Then, the QIP holds for $X$.  In particular, for static conductances, i.e.\ $p' = \infty$ and $q' = \infty$, a QIP holds provided that $p > 1$.  Note that in case $p' \geq p$, the ergodic theorem shows that \eqref{eq:int:Exp:1d} is satisfied whenever $\mean[\om_t(e)^{p'}] < \infty$ and $\mean[\om_t(e)^{-q'}]<\infty$.  
\end{remark}
\begin{remark}
  Based on personal communication with Marek Biskup, we expect that a quenched invariance principle under optimal integrability conditions, namely $\mean[\om_t(e)] < \infty$ and $\mean[\om_t(e)^{-1}] < \infty$, can be proven by adapting the strategy that has been successfully used in the two-dimensional static RCM, cf.\cite[Theorem~4.2]{Bi11}.  In contrast to the static RCM, the harmonic coordinate -- an essential ingredient in the proof -- can not be constructed explicitely for time-dependent conductances in the one-dimensional model.
\end{remark}
The strategy of the proof of the QIP is rather standard and based on \emph{harmonic embedding}, see \cite{Bi11} for a detailed exposition of this method in the static situation.  A key ingredient is to decompose the process $X_t = \Phi(\om, t, X_t) + \chi(\om, t, X_t)$ such that the process $M_t = \Phi(\om, t, X_t)$ is a martingale under $\Prob_{0,0}^{\om}$ with respect to the filtration $\cF_t = \si(X_s, s \leq t)$, where the random function $\Phi\!:\Om \times \bbR \times \bbZ^d \to \bbR^d$, also called \emph{harmonic coordinate}, solves for $\prob$-a.e.\ $\om$ the following parabolic equation
\begin{align}
  \partial_t \Phi(\om, t, x) \,+\, \cL_t^{\om}\Phi(\om, t, x) \;=\; 0,
  \qquad 
  \Phi(\om, 0, 0) \;=\; 0.
\end{align}
The random function $\chi\!:\Om \times \bbR \times \bbZ^d \to \bbR^d$ is also known as the \emph{corrector}.  A QIP for the martingale part can be easily obtained by standard methods.  In order to obtain a QIP for the process $X$, by Slutsky's theorem, it suffices to verify that for any $T > 0$ and $\prob$-a.e.\ $\om$
\begin{align}
  \sup_{0 \leq t \leq T} \big|\tfrac{1}{n}\, \chi(\om, t n^2, X_{t n^2})\big|
  \underset{n \to \infty}{\;\longrightarrow\;}
  0
  \qquad \text{in $\Prob_{0,0}^{\om}$-probability},
\end{align}
which can be deduced from the $\ell^{\infty}$-sublinearity of the corrector:
\begin{align}\label{eq:sublinearity:infty}
  \lim_{n \to \infty} \max_{(t,x) \in Q(n)} \big|\tfrac{1}{n}\, \chi(\om, t, x)\big|
  \;=\;
  0,
  \qquad \prob\text{-a.s.}
\end{align}
The main challenge in the proof of the QIP for random walks among time-dependent random conductances is both the construction of the corrector and to prove \eqref{eq:sublinearity:infty}.  For static conductances the construction of the corrector is based on a simple projection argument exploiting the symmetry of the generator of the \emph{process as seen from the particle}.  However, for time-dependent environments this strategy fails, since the time-space process $\{(t, X_t) : t \geq 0\}$ is not reversible and hence the generator, $\widehat{\cL}$, corresponding to the process $\{\tau_{t, X_t} \om : t \geq 0\}$ is not symmetric with respect to the invariant measure $\prob$.  For this reason, the actual construction of the corrector is more involved, cf.\ \cite[Section~2]{ACDS16} and based on the following argument.  First, by adding of suitable regularisation, the bilinear form associated to $\widehat{\cL}$ is coercive and bounded, and the existence of a regularised version of the corrector is guaranteed by the Lax-Milgram lemma.  The corrector is obtained in a second step by passing to the limit in a suitable sense.

The $\ell^{\infty}$-sublinearity of the rescaled corrector $\frac{1}{n} \chi$ follows from Moser's iteration scheme, which allows to bound $\|\frac{1}{n} \chi\|_{\infty,\infty,Q(n)}$ from above in terms of $\|\frac{1}{n} \chi\|_{1,1,Q(2n)}$.  Thus, the $\ell^{\infty}$-sublinearity can be deduced from the $\ell^1$-sublinearity of the corrector.  The proof of the latter is based on Birkhoff's pointwise ergodic theorem.  One purpose of this note is to present a simplified proof of the $\ell^{1}$-sublinearity in $d=1$ which will greatly ease the corresponding proof in higher dimensions.

Moser's iteration is based on two main ingredients: a Sobolev inequality which allows to control for a suitable  $r > 2$ the $\ell^r(\bbZ^d)$-norm of a function $f$ in terms of the Dirichlet form (cf.\ Lemma~\ref{lemma:sobolev}) and an energy estimate for solutions of a certain class of Poisson equations.  It is well known that Sobolev inequalities can be deduced from the isoperimetric properties of the underlying space.  On $\bbZ^d$, there exists a large variety of Sobolev inequalities. Writing $|\nabla f(x,y)| = |f(x) - f(y)|$ for $\{x,y\} \in E_d$, the form of such inequalities is dimension depending.  More precisely, for any $f\!: \bbZ^d \to \bbR$ with compact support, it holds that for $1 \leq \al < d$
\begin{align}\label{eq:sobolev:alpha<d}
  \norm{\mspace{1mu}f\mspace{1mu}}{\frac{d\al}{d-\al}}{\bbZ^d}
  \;\leq\;
  C(d, \al)\, \norm{\,|\nabla f|\,}{\al}{E_d},
\end{align}
whereas for $\al > d$ the shape of the inequalities changes:
\begin{align}\label{eq:sobolev:d<alpha}
  \norm{\mspace{1mu}f\mspace{1mu}}{\infty}{\bbZ^d}
  \;\leq\;
  C(d, \al)\,
  \norm{\mspace{1mu} f\mspace{1mu}}{\al}{\bbZ^d}^{1 - (d/\al)}\;
  \norm{\,|\nabla f|\,}{\al}{E_d}^{d/\al}.
\end{align}
For the RCM with \emph{uniform elliptic} conductances on $\bbZ^d$ with $d \geq 3$, the Sobolev inequality with $r = r(d) = d/(d-2)$ is immediate from \eqref{eq:sobolev:alpha<d}.  In the case where the conductances are \emph{elliptic}, that is $\om_t(e) \in (0, \infty)$ for all $t \in \bbR$ and $e \in E_d$, starting from \eqref{eq:sobolev:alpha<d} with a suitable $\al < 2 \leq d$, we obtain by means of H\"older's inequality a weighted Sobolev inequality in space with $r = r(d,q) = d/(d-2+d/q)$, see \cite[Proposition~5.4]{ACDS16}.  In the case $d=1$, by using the Cauchy-Schwarz inequality, we immediately obtain a weighted Sobolev inequality from \eqref{eq:sobolev:d<alpha} with $r = \infty$.  
\smallskip

Random motion in random environments has attracted much interest during the past decades.  In particular, the question whether an annealed or quenched invariance principle holds has been studied intensively.  For dynamic random environments, an annealed invariance principle was first shown in \cite{Ma86} for a one-dimensional random walk in a random environment that is i.i.d.\ in space and Markovian in time.  By using analytic, probabilistic and ergodic techniques, annealed and quenched invariance principles have been established by now for various models falling mostly in one of the following two categories: \emph{independent in time} \cite{Be04, BMP97, BMP04, JR11, RS05} or \emph{independent in space and Markovian in time} \cite{BBMP98, BMP00, St04,BZ06, DL09}.  In all these models a \emph {good mixing} behaviour of the environment, i.e. the polynomial decay of time-space correlations, remained a major requirement.

For the RCM with time-space ergodic conductances an annealed and quenched invariance principle has been first proven in \cite{An14} in the uniform elliptic, polynomial mixing case. Recently, the assumptions on the law of the environment has been significantly relaxed. In \cite{ACDS16} a QFLCT has been proven in $d \geq 2$ for the dynamic RCM with elliptic, time-space ergodic conductances satisfying a certain integrability condition.  A similar quenched result in the non-elliptic case for general ergodic environments under suitable moment conditions has been obtained recently in \cite{DGR15} for random walks in time-dependent balanced environments, that is
\begin{align*}
  \om_t(x, x+e_i) \;=\; \om_t(x, x-e_i),
  \qquad \forall\, i = 1, \ldots, d.
\end{align*}
The discrete-time random walk among time-dependent conductance behaves quite differently even in the uniform elliptic case, in particular anomalous heat kernel behaviour occurs, cf.\ \cite{HK15}.
\medskip

The paper is organized as follows: In Section~\ref{sec:harm}, after recalling the construction of the corrector for arbitrary $d \geq 1$, we prove the convergence of the martingale part.  Then, in Section~\ref{sec:sublinearity} we show in dimension $d=1$ that the corrector is sublinear.
\smallskip

Throughout this paper we suppose that Assumption~\ref{ass:P} holds.

\section{Harmonic embedding and the corrector}
\label{sec:harm}
In the sequel, we discuss for any $d \geq 1$ the construction of a corrector $\chi$ to the time-inhomogeneous process $X$ such that $M_t = X_t - \chi(\om, t, X_t)$ is a martingale under $\Prob_{0,0}^{\om}$ for $\prob$-a.e.\ $\om$ , and we prove a quenched invariance principle for the martingale part.
\begin{definition} 
  A measurable function, also called random field, $\Psi\!: \Om \times \bbZ^d \to \bbR^d$ satisfies the (space) cocycle property if for $\prob$-a.e. $\om$
  \begin{align*}
    \Psi(\tau_{0, x\,}\om, y-x)
    \;=\;
    \Psi(\om, y) - \Psi(\om, x),
    \qquad \text{for } x, y \in \bbZ^d.
  \end{align*}
  We denote by $L^2_\mathrm{cov}$ the set of functions $\Psi\!: \Om \times \bbZ^d \to \bbR^d$ satisfying the (space) cocycle property such that
  \begin{align*}
    \Norm{\Psi}{L_\mathrm{cov}^2}^2
    \;\ldef\;
    \mean\!\Big[
      {\textstyle \sum_{x \sim 0}}\, \om_0(0, x)\, |\Psi(\om,x)|^2
    \Big]
    \;<\;
    \infty,
  \end{align*}
  where $| \cdot |$ denotes the usual Euclidean $2$-norm in $\bbR^d$.
\end{definition} 
The position field $\Pi\!: \Om \times \bbZ^d \to \bbR^d$ is defined by $\Pi(\om, x) \ldef x$.  Observe that $\Pi$ is an element of $L_{\mathrm{cov}}^2$, since $\Pi(\om, x+y) - \Pi(\om, x) = \Pi(\tau_{0,x\,} \om ,y)$ for all $\om \in \Om$ and any $x, y \in \bbZ^d$ and $\|\Pi\|_{L_\mathrm{cov}^2} = \mean[\mu_0^{\om}(0)]^{1/2} < \infty$.

We associate to $\vp\!: \Om \to \bbR^d$ a (space) gradient $\mD \vp\!: \Om \times \bbZ^d \to \bbR^d$ defined by
\begin{align*}
  \mD \vp (\om,x)
  \;=\;
  \vp(\tau_{0,x\,} \om) - \vp(\om),
  \qquad x \in \bbZ^d.
\end{align*}
Obviously, if the function $\vp$ is bounded, $\mD \vp$ is an element of $L_{\mathrm{cov}}^2$.  Note that $L_{\mathrm{cov}}^2$ is a Hilbert space.  Further, let us introduce an orthogonal decomposition of the space $L_{\mathrm{cov}}^2 = L_{\mathrm{pot}}^2 \oplus L_{\mathrm{sol}}^2$, where
\begin{align*}
  L_{\mathrm{pot}}^2
  \;=\;
  \mathop{\mathrm{cl}}
  \big\{ \mD \vp \mid \vp\!: \Om \to \bbR\; \text{ bounded} \big\}
  \;\text{ in }\;  L_{\mathrm{cov}}^2,
\end{align*}
being the closure in $L_{\mathrm{cov}}^2$ of the set gradients and $L_{\mathrm{sol}}^2$ be the orthogonal complement of $L_{\mathrm{pot}}^2$ in $L_{\mathrm{cov}}^2$.  Further, set $T_t \vp \ldef \vp \circ \tau_{t,0}$ for $t \in \bbR$ and define the following operator, also called (time) gradient, $D_0\!: \mathop{\mathrm{dom}}(D_0) \subset L^2(\Om, \prob) \to L^2(\Om, \prob)$ by
\begin{align}
  D_0 \vp
  \;\ldef\;
  \lim_{t \to 0}\, \frac{1}{t}\big( T_t \vp \,-\, \vp \big)
\end{align}
where $\mathop{\mathrm{dom}}(D_0)$ is the set of all $\vp \in L^2(\Om, \prob)$ such that the limit above exists.  Notice that $\{T_t \,:\, t \in \bbR\}$ is a strongly continuous contraction group on $L^2(\Om, \prob)$, cf.\ \cite[Section 7.1]{ZKO94}, with infinitesimal generator $D_0$. In particular, $\mathop{\mathrm{dom}}(D_0)$ is dense in $L^2(\Om, \prob)$.  As a consequence, for any $\vp \in \mathop{\mathrm{dom}}(D_0)$ the function $t \mapsto \vp(\tau_{t,0\,} \om)$ is weakly differentiable for $\prob$-a.e.\ $\om$.

A key ingredient of the proof is the existence of a random coordinate system $\Phi(\om, t, x)$ which is known as harmonic coordinates.
\begin{theorem}
  There exists $\Phi_0 \in L_{\mathrm{cov}}^2$ which is characterized by the following properties:
  \begin{enumerate}
  \item[(i)] the function $\chi_0 \ldef \Pi - \Phi_0 \in L_{\mathrm{pot}}^2$; 
  \item[(ii)] (time-space) harmonicity of the function $\Phi\!: \Om \times \bbR \times \bbZ^d \to \bbR^d$,
    \begin{align}\label{eq:Phi:representation}
      \Phi(\om, t, x)
      \;=\;
      \Phi_0(\tau_{t,0\,} \om, x)
      \,-\,
      \int_0^t \big(\cL_s^{\om} \Phi_0(\tau_{s,0\,}\om, \cdot)\big)(0)\, \md s
    \end{align}
    in the sense that $\Phi$ is differentiable for almost every $t \in \bbR$ and
    \begin{align} \label{eq:Phi:pde}
      \partial_t \Phi(\om,t,x) \,+\, \cL^\om_t \Phi(\om,t,x) \;=\; 0,
      \qquad 
      \Phi(\om, 0, 0) \;=\; 0.
    \end{align}
  \end{enumerate}
\end{theorem}
\begin{proof}
  The proof, inspired by an argument given in in \cite{FK99}, is based on an application of the Lax-Milgram Theorem in order to solve in a first step a regularized version of Equation~\eqref{eq:Phi:pde}.  By taking limits in a suitable distribution space, we construct out of the solution to the regularized equation the harmonic coordinate.
  
  For a detailed proof we refer to \cite[Section~2]{ACDS16}.
\end{proof}
\begin{definition}
  The corrector $\chi\!: \Om \times \bbR \times \bbZ^d \to \bbR^d$ is defined as
  \begin{align*}
    \chi(\om, t, x) \;\ldef\; \Pi(\om,x) \,-\, \Phi(\om,t,x).
  \end{align*}
\end{definition}
In the following corollary we summarize properties of $\chi$ and $\chi_0$.
\begin{corro}\label{corro:property:chi}
  Let $\chi_0 \in L_{\mathrm{pot}}^2$ be defined as in the previous theorem.  Then,
  \begin{enumerate}
  \item[(i)] $\chi_0 \in L^1(\prob)$ with $\mean[\chi_0(\om, \hat{e})] = 0$ for all $\hat{e} \in \{\pm e_1, \ldots, e_d\}$;
  \item[(ii)] for $\prob$-a.e.\ $\om$, $t \in \bbR$ and $x \in \bbZ^d$, the corrector can be written as
    \begin{align}\label{eq:chi:split}
      \chi(\om, t, x)
      \;=\;
      \chi_0(\tau_{t,0\,}\om, x)
      \,+\,
      \int_0^t \big(\cL_s^{\om} \Phi_0(\tau_{s,0\,}\om, \cdot)\big)(0)\, \md s.
    \end{align}
  \end{enumerate}
\end{corro}
Define $M_t \ldef \Phi(\om, t, X_t)$ for any $t \geq 0$ and $\om \in \Om$.  In view of \eqref{eq:Phi:pde}, it follows that for $\prob$-a.e.\ $\om$ and any $v \in \bbR^d$ the processes $M = \{M_t \,:\, t \geq 0\}$ and $v \cdot M$ are $\prob_{0,0}^{\om}$-martingales with respect to the filtration $\cF_t = \si(X_s, s \leq t)$.  Moreover, the quadratic variation process of the latter is given by
\begin{align}\label{eq:martingale:QV}
  \langle v \cdot M \rangle_t
  \;=\;
  \int_0^t
  \sum_{y \in \bbZ^d} (\tau_{s, X_s\,}\om)_0(0, y)\,
  \big( v \cdot \Phi_0(\tau_{s,X_s\,} \om, y)\big)^2\,
  \md s.
\end{align}
In the next proposition we show both the convergence of the martingale part and the non-degeneracy of the limiting covariance matrix.
\begin{prop}[QIP for the martingale part]\label{prop:QIP:martingal}
  For $\prob$-a.e.\ $\om$, under $\Prob_{0,0}^{\om}$ the sequence of processes $\{\frac{1}{n} M_{t n^2} \,:\, t \geq 0\}$ converges in law to a Brownian motion with a deterministic, time-independent, non-degenerate covariance matrix $\Si^2$ given by
  \begin{align}
    \Si_{i,j}^2
    \;=\;
    \mean\!%
    \Big[
      {\textstyle \sum_{x \sim 0}}\;
      \om_0(0,x)\, \Phi_0^i(\om, x)\, \Phi_0^j(\om, x)
    \Big].
  \end{align}
\end{prop}
\begin{proof}
  The proof is based on the martingale central limit theorem by Helland (see Theorem 5.1a) in \cite{He82}); the proofs in \cite{ABDH12} or \cite{MP07} can be easily transferred into the time dynamic setting.  The argument relies on the convergence of the quadratic variation of $\{\frac{1}{n} M_{t n^2} \,:\, t \geq 0\}$. Note that the quadratic variation of $M$ is written in terms of the environment process $\{\tau_{t, X_t\,} \om \,:\, t \geq 0\}$ which is a Markov process taking values in $\Om$ with generator $\widehat{\cL}\!: \mathop{\mathrm{dom}}(D_0) \to L^2(\Om,\prob)$,
  \begin{align}
    \big(\widehat{\cL} \vp\big)(\om) 
    \;=\;
    D_0 \phi(\om) \,+\, \sum_{x \sim 0} \om_0(0,x)\,
    \big(\vp(\tau_{0,x\,} \om) \,-\, \vp(\om)\big).
  \end{align}
  Since the measure $\prob$ is invariant and ergodic for the process $\{\tau_{t, X_t\,} \om \,:\, t \geq 0\}$, see \cite[Lemma 2.4]{ADS15} and \cite[Proposition 2.1]{An14} for detailed proofs, the desired convergence of the quadratic variation is a consequence of the ergodic theorem.  Finally, we refer to Proposition 4.1 in \cite{Bi11} for a proof that $\Si^2$ is nondegenerate.
\end{proof} 

\section{Sublinearity of the corrector}
\label{sec:sublinearity}
The key ingredient in the proof of Theorem~\ref{thm:main:1d} is the $\ell^{\infty}$-sublinearity of the corrector as stated in the proposition below.  For simplicity, we focus on the one-dimensional case only; the case $d \geq 2$ has been treated in \cite{ACDS16}.  Recall that $Q(n) = [0, n^2] \times B(n)$ where $B(n) \equiv B(0,n)$.
\begin{prop}[$\ell^{\infty}$-sublinearity]\label{prop:sublinearity:infty}
  For any $p, q \in [1, \infty]$ such that
  \begin{align}
    \frac{1}{p-1} \,+\, \frac{1}{q(p-1)} \;<\; 1.
  \end{align}
  assume that $\mean[\om_t(e)^p] < \infty$ and $\mean[\om_t(e)^{-q}] < \infty$ for all $t \in \bbR$ and $e \in E_1$.  Then,
  \begin{align}
    \lim_{n \to \infty} \max_{(t,x) \in Q(n)}\,
    \big|\tfrac{1}{n}\, \chi(\om, t, x)\big|
    \;=\;
    0,
    \qquad \prob\text{-a.s.}
  \end{align}
\end{prop}
The proof is based on both ergodic theory and purely analytic arguments. In a first step, we show the $\ell^{1}$-sublinearity of the corrector, that is the convergence of $\frac{1}{n} \chi$ to zero in the $\| \cdot \|_{1,1, Q(n)}$-norm.  This proof uses the ergodic theorem, the cocycle-property of $\chi_0$ and the fact that $\Phi$ has the particular representation \eqref{eq:Phi:representation} and solves the equation \eqref{eq:Phi:pde}.  By means of the Moser iteration that allows to establish a maximal inequality for a certain class of Poisson equations, we extend in a second step the $\ell^1$-sublinearity of the corrector to the $\ell^\infty$-sublinearity. 

\subsection{$\ell^1$-sublinearity}
Our main goal in this subsection is to proof that the corrector is sublinear in the following sense:
\begin{prop}[$\ell^1$-sublinearity]\label{prop:sublinearity:1}
  It holds that
  \begin{align}\label{eq:sublinearity:1}
    \lim_{n \to \infty} \frac{1}{n^2}\,
    \int_0^{n^2} \frac{1}{|B(n)|}\, \sum_{x \in B(n)}
    \big|\tfrac{1}{n}\, \chi(\om, t, x) \big|\; \md t
    \;=\;
    0,
    \qquad \prob\text{-a.s.}
  \end{align}
\end{prop}
Our proof of Proposition~\ref{prop:sublinearity:1} relies on the following three lemmas.
\begin{lemma}\label{lemma:chi_0:directional}
  For any $\hat{e} \in \{-1, +1\}$ we have that
  \begin{align}\label{eq:chi_0:directional}
    \lim_{n \to \infty} \frac{1}{n}\, \chi_0(\om, n \hat{e})
    \;=\;
    0,
    \qquad \prob\text{-a.s.}
  \end{align}
\end{lemma}
\begin{proof}
  By rewriting $\chi_0$ as a telescopic sum and using the cocycle property, we first obtain that
  \begin{align}
    \frac{1}{n}\, \chi_0(\om, n\hat{e})
    \;=\;
    \frac{1}{n}\,
    \sum_{j=0}^{n-1} \big(\chi_0(\om, (j+1) \hat{e}) - \chi_0(\om,j \hat{e})\big)
    \;=\;
    \frac{1}{n}\, \sum_{j=0}^{n-1} \chi_0(\tau_{0, j \hat{e}\,}\om, \hat{e}).
  \end{align}
  In view Corollary~\ref{corro:property:chi}(i), the ergodic theorem ensures the existence of the limit
  \begin{align*}
    f_{\hat{e}}(\om)
    \;=\;
    \lim_{n \to \infty} \frac{1}{n}\, \chi_0(\om, n \hat{e}),
    \qquad \prob\text{-a.s. and in } L^1(\Om, \prob).
  \end{align*}
  In particular, $\mean[f_{\hat{e}}] = 0$ and, by construction, $f_{\hat{e}}(\om)$ is invariant with respect to space shift.  Thus, it remains to show that $f_{\hat{e}}(\om) = f_{\hat{e}}(\tau_{t,0\,} \om)$ for any $t \in \bbR$.  But,
  \begin{align}\label{eq:chi_0:shift}
    \chi_0(\tau_{t,0\,} \om, n \hat{e})
    &\;=\;
    \chi(\om, t, n\hat{e}) \,-\, \chi(\om, t, 0)
    \nonumber\\[.5ex]
    &\;=\;
    \chi_0(\om, n \hat{e})
    \,+\,
    \int_0^t
      \big(\cL_s^{\om}\Phi_0(\tau_{s,0\,}\om, \cdot) \big)(n \hat{e})\,
    \md s
    \,-\, \chi_0(\om, 0).
  \end{align}
  Further, notice that $(\cL_s^{\om} \Phi_0(\tau_{s,0\,}\om, \cdot))(n\hat{e}) = (\cL_0^{\om} \Phi_0(\om, \cdot))(0) \circ \tau_{s,n\hat{e}}$ and
  \begin{align*}
    \mean\!\big[\big|(\cL_0^{\om} \Phi_0)(0)\big|]
    \;\leq\;
    \mean[\mu_0^{\om}(0)]^{1/2}\, \Norm{\Phi_0}{L_{\mathrm{cov}}^2}
    \;<\;
    \infty.
  \end{align*}
  Therefore, after dividing both sides of \eqref{eq:chi_0:shift} by $n$, the $L^1(\prob)$-limit of the last two terms vanishes.  Thus, we conclude that $f_{\hat{e}}(\tau_{t,0\,}\om) = f_{\hat{e}}(\om)$ for $\prob$-a.e.\ $\om$, and \eqref{eq:chi_0:directional} follows.
\end{proof}
\begin{lemma}\label{lemma:chi_0:1}
  It holds that
  \begin{align}\label{eq:chi_0:1}
    \lim_{n \to \infty} \frac{1}{n^2}\,
    \int_0^{n^2}\mspace{-6mu} \max_{x \in B(n)}
      \big|\tfrac{1}{n}\, \chi_0(\tau_{t,0\,}\om, x) \big|\,
    \md t
    \;=\;
    0,
    \qquad \prob\text{-a.s.}
  \end{align}
\end{lemma}
\begin{proof}
  Since $\chi_0 \in L_{\mathrm{pot}}^2$, there exists a sequence of bounded functions $\vp_k\!: \Om \to \bbR$ such that $D\vp_k \to \chi_0$ in $L_{\mathrm{cov}}^2$ as $k \to \infty$.  Thus, for any $k \geq 1$ fixed and $x \in B(n)$ we obtain
  \begin{align*}
    \big|\chi_0(\tau_{t,0\,}\om, x) \big|
    &\;\leq\;
    2\|\vp_k\|_{L^{\infty}(\Om,\prob)}
    \,+\,
    \big|(\chi_0-D\vp_k)(\tau_{t,0\,}\om, x) \big|
    \\[.5ex]
    &\;\leq\;
    2\|\vp_k\|_{L^{\infty}(\Om,\prob)}
    \,+\,
    \sum_{j=0}^{n-1}\big|(\chi_0-D\vp_k)(\tau_{t,j \hat{e}\,}\om, \hat{e}) \big|,
  \end{align*}
  where $\hat{e} = \sign{x}$.  Note that we used the cocycle property in the last step.  Hence, by means of Birkhoff's pointwise ergodic theorem, we obtain that for $\prob$-a.e.\ $\om$
  \begin{align*}
    &\lim_{n \to \infty}
    \frac{1}{n^2}\,
    \int_0^{n^2}\mspace{-6mu} \max_{x \in B(n)}
    \big|\tfrac{1}{n}\, \chi_0(\tau_{t,0\,}\om, x) \big|\; \md t
    \\[.5ex]
    &\mspace{36mu}\leq\;
    \sum_{\hat{e} \in \{-1,+1\}} \mspace{-10mu}
    \mean\!\big[\big|(\chi_0-D\vp_k)(\om, \hat{e})\big|\big]
    \;\leq\;
    \mean[\nu_0^{\om}(0)]^{1/2}\, \Norm{\chi_0 - D\vp_k}{L_{\mathrm{cov}}^2}.
  \end{align*}
  By taking the limit $k \to \infty$, the assertion \eqref{eq:chi_0:1} follows.
\end{proof}
\begin{lemma}\label{lemma:chi:t}
  We have that
  \begin{align}\label{eq:chi:t}
    \lim_{n \to \infty} \frac{1}{n^2}\,
    \int_0^{n^2}\mspace{-6mu} \big|\tfrac{1}{n}\, \chi(\om, t, 0) \big|\; \md t
    \;=\;
    0,
    \qquad \prob\text{-a.s.}
  \end{align}
\end{lemma}
\begin{proof}
  The proof of \eqref{eq:chi:t} comprises two steps.\\
  \textsc{Step 1:} For any $\vp \in L^1(\Om, \prob)$ and $g\!: \bbR \to \bbR$ bounded and compactly supported with $\int_{\bbR} g(y)\, \md y = 0$ an extension of Birkhoff's ergodic theorem, cf.\ \cite{BD03}, yields
  \begin{align*}
    F^{\om}(k)
    \;\ldef\;
   \frac{1}{k^3}
    \int_0^{k^2}\mspace{-3mu}
      \sum_{x \in \bbZ}\, g(x/k)\, \vp(\tau_{s,y\, }\om)\,
    \md s
    \underset{k \to \infty}{\;\longrightarrow\;}
    \bigg(\int_{\bbR}\! g(y)\, \md y\bigg)\, \mean\!\big[\vp\big]
    \;=\;
    0
  \end{align*}
  $\prob$-a.s.  Hence, for every $\ve \in (0, 1)$ there exists $N^{\om}(\ve) < \infty$ such that $|F^{\om}(k)| < \ve$ for all $k \geq N^{\om}(\ve)$.  Since
  \begin{align*}
    \sup_{k > 0}\,
    \frac{1}{k^3}
    \int_0^{k^2}\mspace{-3mu}
      \sum_{x \in \bbZ}\, |g(x/k)|\,|\vp(\tau_{s,y\, }\om)|\,
    \md s
    \;\leq\;
    M
    \;<\;
    \infty,
  \end{align*}
  by choosing $n \geq N^{\om}(\ve)\sqrt{M/\ve}$ we get 
  \begin{align*}
    \frac{1}{n^2}
    \int_0^{n^2}\mspace{-6mu}
      \big| F^{\om}(\sqrt{t})\big|\,
    \md t
    \;\leq\;
    \Big(\frac{N^{\om}(\ve)}{n}\Big)^2\, M \,+\, \ve
    \;\leq\;
    2 \ve.
  \end{align*}
  Thus, we conclude that for $\prob$-a.e.\ $\om$
  \begin{align}\label{eq:chi:t:step1}
    \lim_{n \to \infty} \frac{1}{n^2}
    \int_0^{n^2}\mspace{-2mu}
      \bigg|
        \frac{1}{t^{3/2}}
        \int_0^t
          \sum_{x \in \bbZ}\, g(x/k)\, \vp(\tau_{s,y\, }\om)\,
        \md s
      \bigg|\,
    \md t
    \;=\;
    0.
  \end{align}
  \textsc{Step 2:} Let us now prove \eqref{eq:chi:t}.  First, Corollary~\ref{corro:property:chi}(ii) and \eqref{eq:Phi:pde} imply that 
  \begin{align}\label{eq:chi:t:rewritten}
    \chi(\om, t, 0)
    &\;=\;
    \chi_0(\om, y) \,+\, \int_0^t \partial_s \Phi(\om, s, y)\, \md s
    \,-\, \chi_0(\tau_{t,0\,}\om, y)
    \nonumber\\[.5ex]
    &\;=\;
    \chi_0(\om, y) \,-\, \chi_0(\tau_{t,0\,}\om, y)
    \,+\, \int_0^t \big(\cL_s^{\om} \Phi_0(\tau_{s,0\,}\om, \cdot)(y)\, \md s.
  \end{align}
  for any $y \in \bbZ$.  Further, consider the function $f\!:\bbR \to [0,1]$, $x \mapsto [1 - |x|]_+$ and set $f_t(x) \ldef f(x/\sqrt{t})$ for any $t > 0$.
  Notice that $\supp f_t \subset B(\sqrt{t})$ and $\sum_{z \in \bbZ} f_t(x) \geq \sqrt{t}/2$.  Then, by multiplying both sides of \eqref{eq:chi:t:rewritten} with $f_t$ and summing over all $y \in \bbZ$ we obtain that, for any $t \in (0, n^2]$,
  \begin{align*}
    \big|\tfrac{1}{n}\,\chi(\om, t, 0)\big|
    &\;\leq\;
    \max_{x \in B(n)} \big|\tfrac{1}{n}\, \chi_0(\om, x)\big|
    \,+\, \max_{x \in B(n)} \big|\tfrac{1}{n}\, \chi_0(\tau_{t,0\,}\om, x)|
    \nonumber\\[.5ex]
    &\mspace{36mu}+\,
    \frac{2\sqrt{t}}{n} \sum_{z \sim 0}\,
    \bigg|
      \frac{1}{t^{3/2}}
      \int_0^t
        \sum_{y \in \bbZ}\, g_z(y/\sqrt{t})\, \vp_z(\tau_{s,y\,} \om)\,
      \md s
    \bigg|,
  \end{align*}
  where we introduced for $z \sim 0$ the functions $g_z(y) \ldef \sqrt{t} \big(f(y + z/\sqrt{t}) - f(y)\big)$ and $\vp_z(\om) \ldef \om_0(0,z)\, \Phi_0(\om,z)$ to lighten notation.  Since
  \begin{align*}
    \mean[\om_0(0,z)\, \Phi_0(\om,z)]
    \;\leq\;
    \mean[\om_0(0,z)]^{1/2}\, \Norm{\Phi_0}{L_{\mathrm{cov}}^2}
    \;<\;
    \infty,
  \end{align*}
  $\vp_z \in L^1(\Om, \prob)$.  Moreover, the function $g_z$ is bounded, compactly supported with $\int_{-1}^1 g_z(y)\, \md y = 0$.  In particular, $\supp g_z \subset [-2,2]$ for all $t \geq 1$ and any $z \sim 0$.  Thus, in view of Lemma~\ref{lemma:chi_0:directional} and \ref{lemma:chi_0:1} together with \eqref{eq:chi:t:step1} the assertion \eqref{eq:chi:t} follows.
\end{proof}
\begin{proof}[Proof of Proposition~\ref{prop:sublinearity:1}]
  Since $\chi_0 \in L_{\mathrm{pot}}^2$, we have that $\chi_0(\tau_{t,0}\om, 0) = 0$ for any $t \in \bbR$.  Hence, \eqref{eq:chi:split} can be rewritten as $\chi(\om, t, x) = \chi_0(\tau_{t, 0\,} \om, x) + \chi(\om, t, 0)$ for any $t \in \bbR$ and $x \in \bbZ$ and $\prob$-a.e.\ $\om$.  Thus, \eqref{eq:sublinearity:1} follows from \eqref{eq:chi_0:1} and \eqref{eq:chi:t}.
\end{proof}

\subsection{$\ell^{\infty}$-sublinearity}
The next proposition relies on the application of the Moser iteration scheme that has been implemented for general graphs in \cite[Section~5.2]{ACDS16}.  A key ingredient in this approach is the following Sobolev inequality.
\begin{lemma}[local space-time Sobolev inequality]\label{lemma:sobolev}
  Let $Q = I \times B$, where $I \subset \bbR$ is a compact interval and $B \subset \bbZ$ is finite and connected.  Then, for any $q' \in [1, \infty]$ and $u:\bbR \times \bbZ \to \bbR$ with $\supp u_t \subset B$, it holds that
  \begin{align}\label{eq:sob:ineq:2}
    \Norm{u^2}{\infty, q'/(q'+ 1), Q}
    \;\leq\;
    |B|^2\, \Norm{\nu^{\om}}{1, q'\!, Q}\;
    \bigg(
      \frac{1}{|I|}\, \int_I\, \frac{\cE_{t}^{\om}(u_t)}{|B|}\; \md t
    \bigg),
  \end{align}
  where for any $f\!: \bbZ \to \bbR$
  \begin{align*}
    \cE_t^{\om}(f)
    \;\ldef\;
    \sum_{\{y, y'\} \in E_1}\mspace{-6mu} \om_t(y,y')\, \big(f(y) - f(y')\big).
  \end{align*}
\end{lemma}
\begin{proof}
  Since $\supp u_t \subset B$, write for any $x \in B$ the function $u_t(x)$ as a telescopic sum and apply the Cauchy-Schwarz inequality.  This yields
  \begin{align*}
    |u_t(x)|^2
    \;\leq\;
    |B|\, \bigg(\frac{1}{|B|} \sum_{y \in B} \nu_t^{\om}(y)\bigg)
    \bigg(\frac{\cE_t^{\om}(u_t)}{|B|}\bigg).
  \end{align*}
  Thus, for any $q' \geq 1$ the assertion follows by H\"older's inequality.
\end{proof}
\begin{prop}[maximal inequality]\label{prop:maximal:inequality}
  Let $p, p', q' \in [1, \infty]$ be such that
  \begin{align}\label{eq:cond:1d:'}
    \frac{1}{p} \cdot \frac{p'}{p'-1} \cdot \frac{q'+1}{q'}
    \;<\;
    1.
  \end{align}
  Then, for every $\al > 0$ there exist $\ka, \ga' > 0$ and $c(p,p',q') < \infty$ such that for $\prob$-a.e.\ $\om$
  \begin{align}\label{eq:maximal:ineq}
    \max_{(t,x) \in Q(n)} \big|\tfrac{1}{n}\, \chi(\om, t, x)\big|
    \;\leq\;
    c\,
    \Big(
      1 \vee \Norm{\mu^{\om}}{p,p',Q(2n)}\, \Norm{\nu^{\om}}{1,q',Q(2n)}
    \Big)^{\!\ka'}\, \Norm{\tfrac{1}{n} \chi(\om, \cdot)}{\al,\al,Q(2n)}.
  \end{align}
\end{prop}
\begin{proof}
  From definition of $\chi$ and \eqref{eq:Phi:pde} it follows that the corrector, $\frac{1}{n} \chi$, is differentiable in $t$ for almost every $t \in \bbR$ and satisfies the following Poisson equation
  \begin{align}\label{eq:poisson}
    \partial_t u \,+\, \cL_t^{\om} u
    \;=\;
    \nabla^* V_t^{\om}
  \end{align}
  with $V_t^{\om}(x, y) \ldef \frac{1}{n} \om_t(x, y) (y-x)$ and $\nabla^* V_t^{\om}(x) = \sum_{y \sim x} V_t^{\om}(x,y)$.  The proof of \eqref{eq:maximal:ineq} is based on the Moser iteration scheme, and the assertion for $\frac{1}{n} \chi$ follows line by line from the proof of \cite[Theorem~5.5]{ACDS16} with $\si=1$, $\si'=1/2$, $n$ replaced by $2n$ if we use instead of \cite[Proposition~5.4]{ACDS16} the Sobolev inequality \eqref{eq:sob:ineq:2}.  For the convenience of reader, we will explain in the sequel the first crucial step of this iteration. 

  For $\si \in (0,1)$, let $Q(\si n) = I(\si n) \times B(\si n)$, where $I(\si n) \ldef [0, \si n^2]$.  Further, consider a cut-off function $\eta\!: \bbZ \to [0, 1]$ with $\supp \eta \subset B(n)$, $\eta \equiv 1$ on $B(\si n)$ and $\max_{\{x,y\} \in E_1} |\eta(x) - \eta(y)| \leq c/n$.  Further, set $u(t,x) \ldef \tfrac{1}{n} \chi(\om,t,x)$ to lighten notation.

  Since $u$ solves \eqref{eq:poisson}, by \cite[Lemma~5.6]{ACDS16}, the following energy estimate holds: For any $\al \geq 1$ and $p, p', \in [1, \infty]$ with $p_*$ and $p_*'$ being the corresponding H\"older conjugates of $p$ and $p'$ we have
  \begin{align}\label{eq:energy:estimate}
    \Norm{|u|^{2\al}}{1, \infty, Q(\si n)} \,+\,
    \int_{I(\si n)}\mspace{-6mu}
      \frac{\cE_{t}^{\om}(\eta |u_t|^{\al})}{|B(n)|}\,
    \md t
    \;\leq\;
    C_{\al} \Norm{\mu^{\om}}{p, p', Q(n)}\,
    \Norm{|u|^{2\al}}{p_*, p_*',Q(n)}^{\ga/(2\al)}
  \end{align}
  where $\ga = 1$ if $\||u|^{2\al}\|_{p_*, p_*',Q(n)} \geq 1$ and $\ga = 1-1/\al$ otherwise.  On the other hand, by means of H\"older's and Young's inequality, cf.\ \cite[Lemma~5.3]{ACDS16}, we obtain for $\al = 1/p_* + q'/(p_*(q'+1))$ that
  \begin{align}\label{eq:interpolation}
    \Norm{|u|^{2\al}}{\al p_*, \al p_*', Q(\si n)}
    \;\leq\;
    \Norm{|u|^{2\al}}{1, \infty, Q(\si n)}
    \,+\, \Norm{|u|^{2\al}}{\infty, q'/(q'+1), Q(\si n)}
  \end{align}
  Thus, combining \eqref{eq:interpolation} with the local space-time Sobolev inequality \eqref{eq:sob:ineq:2} and using that $|B(\si n)|^2/|I(\si n)| \leq 2$ yields
  \begin{align*}
    \Norm{u^2}{\al^2 p_*, \al^2 p_*', Q(\si n)}
    \;\leq\;
    \Big(
      C_{\al} 1 \vee \Norm{\mu^{\om}}{p, p', Q(n)}\, \Norm{\nu^{\om}}{1,q',Q(n)}
    \Big)^{\!1/\al}\, \Norm{u^2}{\al p_*, \al p_*', Q(n)}^{\ga/2}.
  \end{align*}
  Finally, notice that the condition \eqref{eq:cond:1d:'} implies that $\al > 1$.  By iterating this estimate (for details see the proof of \cite[Theorem~5.5]{ACDS16}) the assertion follows. 
\end{proof}
By assumption, $\mean[\om_t(e)^p]< \infty$ and $\mean[\om_t(e)^{-q}] < \infty$ for any $p, q \in [1, \infty]$ satisfying the condition \eqref{eq:cond:1d}.  Since $\|\nu^{\om}\|_{1,q',Q(2n)} \leq \|\nu^{\om}\|_{q',q',Q(2n)}$ for any $q' \in [1, \infty]$ by Jensen's inequality, Birkhoff's ergodic theorem implies that for $\prob$-a.e.\ $\om$
\begin{align*}
  \lim_{n \to \infty} \Norm{\mu^{\om}}{p, p, Q(2n)}
  \;<\; 
  \infty
  \qquad \text{and} \qquad
  \lim_{n \to \infty} \Norm{\nu^{\om}}{q, q, Q(2n)}
  \;<\; 
  \infty.
\end{align*}
Thus, Proposition~\ref{prop:sublinearity:infty} follows immediately from Proposition~\ref{prop:maximal:inequality} with the choice $\al = 1$, $p'=p$ and $q' = q$, combined with Proposition~\ref{prop:sublinearity:1}.
\begin{proof}[Proof of Theorem~\ref{thm:main:1d}]
  Proceeding as in the proof of \cite{ADS15}, the $\ell^{\infty}$-sublinearity of the corrector that we have established in Proposition~\ref{prop:sublinearity:infty} implies that for any $T > 0$ and $\prob$-a.e.\ $\om$
  \begin{align*}
    \sup_{0 \leq t \leq T} \big|\tfrac{1}{n}\, \chi(\om, t n^2, X_{t n^2}) \big|
    \underset{n \to \infty}{\;\longrightarrow\;}
    0
    \qquad \text{in $\Prob_{0,0}^{\om}$-probability}.
  \end{align*}
  Thus, the assertion of Theorem~\ref{thm:main:1d} now follows from Proposition~\ref{prop:QIP:martingal}.
\end{proof}

\subsubsection*{Acknowledgment}
We thank Marek Biskup for bringing our attention to the one-dimensional case.  J.-D.D. thanks RIMS for the kind hospitality and the financial support.

\bibliographystyle{plain}
\bibliography{literature}

\end{document}